\documentclass{article}
\usepackage{amsfonts}
\usepackage{amsmath}

\newtheorem{theorem}{Theorem}

\newtheorem{corollary}[theorem]{Corollary}

\newtheorem{lemma}[theorem]{Lemma}

\newtheorem{remark}[theorem]{Remark}

\newenvironment{proof}[1][Proof]{\noindent\textbf{#1.} }{\ \rule{0.5em}{0.5em}}

\begin{document}

\title{Novel Formulas for $B$-Splines, Bernstein Basis Functions and special numbers: Approach to Derivative and Functional Equations of Generating Functions}
\author{Yilmaz Simsek\\
Department of Mathematics, Faculty of Science, \\ University of Akdeniz, Antalya TR-07058, Turkey; \\ ysimsek@akdeniz.edu.tr}
\maketitle

\begin{abstract}
One of the main purposes of this article is to give functional
equations and differential equations between Bernstein basis functions and
generating functions of $B$-spline curves. Using these equations, very
useful formulas containing the relationships among the uniform $B$-spline curves,
the Bernstein basis functions, and other special numbers and polynomials are
derived. By applying $p$-adic integrals to these polynomials, many novel
formulas are also derived. Furthermore, by applying the partial differential
derivative operator and Laplace transformation to these generating
functions, with aid of higher derivative differential, not only recurrence
relations and the higher derivative formula for $B$-spline curves, but also
infinite series representations are given.
\end{abstract}

\section{Introduction}

Spline theory has been among the most popular areas of mathematics and other
applied sciences in recent years. Isaac Jacob Schoenberg \cite{Schonberg}, a
Romanian-American mathematician, is credited with the invention of splines
and is also known as their father. The best known of these is $B$-spline.
Generally, $B$-splines and the Bezier curves expressed in terms of the
Bernstein basis functions are also known to be used for curve fitting and
numerical differentiation of experimental data. Moreover, they are often
used effectively in computer-aided design and computer graphics (cf. \cite{Goldman,GSS,Farouki,lopez,amc,Salomon}; see
also the references cited in each of these earlier works). The generating
function for the $B$-spline was firstly constructed by Goldman \cite{Goldman}. The motivation of this article is to
give new formulas and finite sums that include $B$-splines and special
polynomials by blending the generating functions with their functional
equations for the Bernstein basis functions, the $B$-spline other special
numbers and polynomials involving special functions, the Apostol type
Bernoulli numbers and polynomials, the Apostol type Euler numbers and
polynomials, the Eulerian numbers and polynomials, the Stirling numbers.

The remaining parts of this article the following notations are used: 
Let $\mathbb{N}$, $\mathbb{Z}$, $\mathbb{Q}$, $\mathbb{R}$ and $\mathbb{C}$
demonstrate the set of natural numbers, the set of integers, the set of
rational numbers, the set of real numbers and the set of complex numbers,
respectively. $\mathbb{N}_{0}=\mathbb{N\cup }\left\{ 0\right\} $. Let $%
\mathbb{C}_{p}$ be the field of $p$-adic completion of algebraic closure of $%
\mathbb{Q}_{p}$, set of $p$-adic rational numbers. Let $\mathbb{Z}_{p}$ be
set of $p$-adic integers.

The remaining parts of this article are also briefly summarized as follows:

In preliminaries section, we give generating functions with their some
properties of special numbers and polynomials.

In Section \ref{section2}, by using generating functions with their functional equations
and using $p$-adic integrals, we give some novel computation formulas of the
Apostol-Bernoulli polynomials, the Apostol-Euler polynomials, the Eulerian
numbers and polynomials, and the Euler Frobenius numbers and polynomials.

In Section \ref{section3}, with the aid of generating functions with their derivative and functional equations, we give many novel identities and relations for
the uniform $B$-splines and the Bernstein basis functions. With the aid of
application of Laplace transform to generating functions for the uniform $B$%
-splines, series representations of the uniform $B$-spline and the Bernstein
basis functions are given.

Finally, this article is concluded with the conclusion section.
\subsection{Preliminaries} \label{Preliminaries}

The Stirling numbers of the second kind are defined by%
\begin{equation}
	\left( e^{t}-1\right) ^{c}=\sum_{n=0}^{\infty }c!S_{2}(n,c)\frac{t^{n}}{n!},
	\label{Sn2}
\end{equation}%
where $c\in \mathbb{N}_{0}$ (cf. \cite{Comtet,SimsekMTJPAM,SimsekFPTA,SrivatavaChoi}).

The array polynomials are defined by%
\begin{equation}
	e^{t\omega}\left( e^{t}-1\right) ^{c}=\sum_{n=0}^{\infty }c!S_{c}^{n}(\omega)\frac{%
		t^{n}}{n!},  \label{arP}
\end{equation}%
where $c\in \mathbb{N}_{0}$ (cf. \cite{Comtet,SimsekMTJPAM,SimsekFPTA,SrivatavaChoi}).

The Apostol-Bernoulli numbers $\mathcal{B}_{n}(\rho )$ are defined by%
\begin{equation}
	K_{n}\left( t;\rho \right) =\frac{t}{\rho e^{t}-1}=\sum_{n=0}^{\infty }%
	\mathcal{B}_{n}(\rho )\frac{t^{n}}{n!}  \label{ApstlBernNumbr}
\end{equation}%
(cf. \cite{Apostol,Bayad,Boyadzhiev,SimsekMTJPAM,SimsekFPTA,SrivatavaChoi}).

\begin{lemma}
	Let $n\in \mathbb{N}_{0}$. Then we have%
	\begin{equation}
		\mathcal{B}_{n}(\rho )=\frac{n\rho }{\left( \rho -1\right) ^{n}}%
		\sum\limits_{s=0}^{n-1}\left( -1\right) ^{s}s!\rho ^{s-1}\left( \rho
		-1\right) ^{n-1-s}S_{2}\left( n-1,s\right)  \label{B1}
	\end{equation}
	\textup{(cf. \cite{Apostol})}.
\end{lemma}

The Apostol-Bernoulli polynomials $\mathcal{B}_{n}(\omega ;\rho )$  are
defined by%
\begin{equation}
	K_{p}\left( \omega ,t;\rho \right) =K_{n}\left( t;\rho \right) e^{t\omega
	}=\sum_{n=0}^{\infty }\mathcal{B}_{n}(\omega ;\rho )\frac{t^{n}}{n!}
	\label{ABP}
\end{equation}%
where $\rho \neq 1$ (cf. \cite{Apostol}).

Substituting $\rho =1$ into (\ref{ABP}), we have the Bernoulli polynomials:%
\begin{equation*}
	B_{n}\left( \omega \right) =\mathcal{B}_{n}\left( \omega ;1\right) 
\end{equation*}
when $\omega=0$, we also have $B_{n}=B_{n}(0)$ (cf. \cite{Apostol,Kilar,Luo,luo1,luo2,SimsekMTJPAM}; see
also the references cited in each of these earlier works).

Using (\ref{ABP}) and (\ref{ApstlBernNumbr}), we have%
\begin{equation*}
	\alpha \mathcal{B}_{1}(1;\alpha )-\mathcal{B}_{1}(\alpha )=1
\end{equation*}%
and for $n\geq 2$,%
\begin{equation*}
	\alpha \mathcal{B}_{n}(1;\alpha )-\mathcal{B}_{n}(\alpha )=0,
\end{equation*}%
\begin{equation}
	\mathcal{B}_{n}(\omega;\alpha )=\sum\limits_{j=0}^{n}\binom{n}{j}\omega^{n-j}\mathcal{B%
	}_{j}(\alpha )  \label{Ap-i0}
\end{equation}%
and%
\begin{equation}
	\alpha \mathcal{B}_{n}(\omega+1;\alpha )\mathcal{=B}_{n}(\omega;\alpha )+n\omega^{n-1}
	\label{ABp-i}
\end{equation}%
(cf. \cite{Apostol,Bayad,Boyadzhiev,SimsekMTJPAM,SrivatavaChoi}).
Let $\phi $ be a complex numbers with $\phi \neq 1$. The Euler Frobenius
polynomials $H_{n}(\omega ;\phi )$ are defined by%
\begin{equation}
	F_{P}(t,\omega ,\phi )=\frac{1-\phi }{e^{t}-\phi }e^{t\omega
	}=\sum_{n=0}^{\infty }H_{n}(\omega ;\phi )\frac{t^{n}}{n!}  \label{FEp}
\end{equation}%
(cf. \cite{Schonberg,SimsekMTJPAM,SrivatavaChoi}).
Substituting $\omega =0$ into (\ref{FEp}), we have%
\begin{equation}
	F_{N}(t,\phi )=\frac{1-\phi }{e^{t}-\phi }=\sum_{n=0}^{\infty }H_{n}(\phi )%
	\frac{t^{n}}{n!}.  \label{FEn}
\end{equation}%
By using (\ref{FEn}), we have%
\begin{equation*}
	H_{n}(\phi )=\left\{ 
	\begin{array}{cc}
		1 & \text{for }n=0 \\ 
		\frac{1}{\phi }\sum\limits_{j=0}^{n}\left( 
		\begin{array}{c}
			n \\ 
			j%
		\end{array}%
		\right) H_{j}(\phi ) & \text{for }n>0.%
	\end{array}%
	\right. 
\end{equation*}%
Substituting $\phi =-1$ into (\ref{FEn}), we have the Euler numbers $%
E_{n}=H_{n}(-1)$ (cf. \cite{Comtet,SimsekFPTA,SrivatavaChoi}).
The Apostol-Euler polynomials of the first kind $\mathcal{E}_{n}(\omega;\rho )$ and the Apostol-Euler numbers of the first kind $\mathcal{E}_{n}(\rho )$
are defined by means of the following generating functions: 
\begin{equation}
	F_{P1}(t,\omega ;\rho )=\frac{e^{t\omega }}{\rho e^{t}+1}=\sum_{n=0}^{\infty
	}\frac{\mathcal{E}_{n}(\omega ;\rho )}{2}\frac{t^{n}}{n!} \label{Cad3}
\end{equation}
and
\begin{equation}
	\frac{1}{\rho e^{t}+1}=\sum_{n=0}^{\infty }\frac{\mathcal{E}_{n}(\rho )}{2}%
	\frac{t^{n}}{n!}.  \label{aen}
\end{equation}

By combining (\ref{Cad3}) with (\ref{aen}), we have%
\begin{equation}
	\mathcal{E}_{n}(\omega ;\rho )=\sum\limits_{j=0}^{n}\binom{n}{j}\omega ^{n-j}%
	\mathcal{E}_{j}(\rho )  \label{AAEN}
\end{equation}%
(cf. \cite{Bayad,Boyadzhiev,SimsekMTJPAM,SrivatavaChoi}).

Substituting $\rho =1$ into (\ref{Cad3}), we have the Euler polynomials of
the first kind:%
\begin{equation*}
	E_{n}\left( \omega \right) =\mathcal{E}_{n}\left( \omega ;1\right) ,
\end{equation*}%
when $\omega=0$, we also have $E_{n}=E_{n}(0)$ (cf. \cite{Apostol,Bayad,Boyadzhiev,Comtet,Kilar, Riardon,SimsekMTJPAM,SrivatavaChoi}).
The Euler Frobenius type polynomials (the Eulerian polynomials) are defined by
\begin{equation}
	F_{a}(t,\rho )=\frac{1}{1-\rho e^{t(1-\rho )}}=\sum_{n=0}^{\infty }\frac{%
		A_{n}(\rho )}{1-\rho }\frac{t^{n}}{n!},  \label{en-1}
\end{equation}%
where $\rho \neq 1$, $A_{n}(\rho )$ is a polynomial in $\rho $\ of degree $%
n-1$ for $n>0$:%
\begin{equation}
	A_{n}(\rho )=\sum\limits_{j=0}^{n}A_{n,j}\rho ^{j},  \label{AnPl}
\end{equation}%
where $A_{n,j}$ is an integer numbers, known as the Eulerian numbers:%
\begin{equation}
	A_{n,j}=\sum\limits_{v=0}^{j}(-1)^{v}\binom{n+1}{v}(j-v)^{n},  \label{EnFor}
\end{equation}%
$j=1,2,\dots,n$, $0\leq j<n$, $n\in \mathbb{N}$ (cf. \cite{Butzer Hauss,CarlitzRos,Comtet}). The Worpitzky identity for the Eulerian
numbers is given as follows:%
\begin{equation}
	\omega ^{n}=\sum\limits_{v=0}^{n}\binom{\omega +v-1}{n}A_{n,v},  \label{w}
\end{equation}%
where $\omega\in \mathbb{R}$, $n\in \mathbb{N}_{0}$ (cf. \cite{Butzer
	Hauss,CarlitzRos}).

These numbers have many interesting combinatoric applications. For instance,
let $A=\left\{ 1,2,3,\ldots ,n\right\} $ and $P(n,j)$ denote the number of
permutation of the elements of the set $A$ that show exactly $j$ increases
between adjacent elements, the first element always being counted as a jump.
Thus, it is known that%
\begin{equation*}
	P(n,j)=A_{n,j}
\end{equation*}%
(cf. \cite{CarlitzRos,Riardon}). In the literature, the
numbers $A_{n,j}$ are represented by many different notations. Some of these
are presented as follows: $A_{n,j}$, $A(n,k)$, $\left\langle 
\begin{array}{c}
n \\ 
j%
\end{array}%
\right\rangle $, $E(n,k)$, $W_{n,k}$ (cf. \cite{Comtet,Butzer Hauss,Riardon,Petersen,Schonberg,SimsekASCM2013,SimsekFPTA,SimsekAADM2018}).
By using umbral calculus method in (\ref{en-1}), we get%
\begin{equation*}
	A_{n}(\rho )=\rho \sum\limits_{j=0}^{n}\binom{n}{j}\left( 1-\rho \right)
	^{n-j}A_{j}(\rho ),
\end{equation*}%
where $A_{0}(\rho )=1$.

Generating functions for the Bernstein basis functions $B_{j}^{k}(\omega)$ are
given by%
\begin{equation}
	f_{\mathbb{B},d}(t,\omega )=(\omega t)^{d}e^{t(1-\omega
		\omega)}=\sum\limits_{k=d}^{\infty }d!B_{d}^{k}(\omega )\frac{t^{k}}{k!},
	\label{BBf}
\end{equation}%
where%
\begin{equation}
	B_{d}^{k}(\omega )=\binom{k}{d}\omega ^{d}(1-\omega )^{k-d}  \label{cBBF}
\end{equation}%
$0\leq d\leq k,$ and $d,k\in \mathbb{N}_{0}$, if $d>k$, then%
\begin{equation*}
	B_{d}^{k}(\omega)=0
\end{equation*}%
(cf. \cite{SimskAcikgoz,SimsekFPABERNST}). In recent years,
many articles have been published covering the generating functions of
Bernstein base functions and their applications in many different fields.
Some of these studies are (cf. \cite{Acikgoz,amc,GSS,BuketAADM,BuketFilomat,SimskAcikgoz}). The
Bernstein polynomials and their properties are available in \cite{lopez} and 
\cite{Salomon} in a very efficient and practical way.

Goldman \cite{Goldman} constructed the following generating function for the
uniform $B$-splines from $N_{0,n}(\omega ;p)$:%
\begin{eqnarray}
	G_{0}(\omega ,t;p) &=&\sum\limits_{j=0}^{p}(-1)^{j}\left( \frac{(\omega
		-j)^{j}t^{j}}{j!}+\frac{(\omega -j)^{j-1}t^{j-1}}{(j-1)!}\right) e^{(\omega
		-j)t}  \label{G-1} \\
	&=&\sum\limits_{n=0}^{\infty }N_{0,n}(\omega ;p)t^{n},  \notag
\end{eqnarray}%
$p\leq \omega\leq p+1$.

In the literature, the uniform $B$-splines $N_{0,n}(\omega ;p)$ are
represented by many different notations. Some of these are presented as
follows: $N_{0,n}(\omega )$, $B_{k}^{n}(\omega )$, $B_{n}(\omega )$ (
cf. \cite{Goldman,Farouki,Salomon,Schonberg}).

By using (\ref{G-1}), using generating function method, Goldman \cite[%
Theorem 3]{Goldman} proved the following well-known Schoenberg's identity
and de Boor recurrence for the uniform $B$-splines, respectively:%
\begin{equation}
	N_{0,n}(\omega ;p)=\frac{1}{n!}\sum\limits_{j=0}^{p}(-1)^{j}\binom{n+1}{j}%
	(\omega -j)^{n}  \label{Bs}
\end{equation}%
where $p\leq \omega \leq p+1$ and%
\begin{equation*}
	N_{0,n}(\omega ;p)=\frac{\omega }{n}N_{0,n-1}(\omega ;p)+\frac{n+1-\omega }{n%
	}N_{1,n-1}(\omega ;p)
\end{equation*}%
(cf. \cite[Theorem 3]{Goldman}).

By using (\ref{Bs}), we values of the basis $N_{0,n}(\omega;p )$ are given as follows:

$N_{0,n}(\omega ;0)=\frac{1}{n!}\omega ^{n}$, $0\leq \omega \leq 1$

$N_{0,n}(\omega ;1)=\frac{1}{n!}\left( \omega ^{n}-\binom{n+1}{1}(\omega
-1)^{n}\right) $, $1\leq \omega \leq 2$

$N_{0,n}(\omega ;2)=\frac{1}{n!}\left( \omega ^{n}-\binom{n+1}{1}(\omega
-1)^{n}+\binom{n+1}{2}(\omega -2)^{n}\right) $, $1\leq \omega \leq 2$

$N_{0,n}(\omega ;3)=\frac{1}{n!}\left( \omega ^{n}-\binom{n+1}{1}(\omega
-1)^{n}+\binom{n+1}{2}(\omega -2)^{n}-\binom{n+1}{3}(\omega -3)^{n}\right) $%
, $1\leq \omega \leq 2$.\\

For $n=1$, we have

$N_{0,1}(\omega ;0)=\omega $, $0\leq \omega \leq 1$

$N_{0,1}(\omega ;1)=-\omega +2$, $1\leq \omega \leq 2$

$N_{0,1}(\omega ;p)=0$, $p\leq \omega \leq p+1$ if $p\geq 2$. \\

For $n=2$, we have

$N_{0,2}(\omega ;0)=\frac{1}{2}\omega ^{2}$, $0\leq \omega \leq 1$

$N_{0,2}(\omega ;1)=\frac{1}{2}(-2\omega ^{2}+6\omega -3)$, $1\leq \omega
\leq 2$

$N_{0,2}(\omega ;2)=\frac{1}{2}(\omega ^{2}-6\omega +9)$, $2\leq \omega \leq
3$

$N_{0,2}(\omega ;p)=0$, $p\leq \omega \leq p+1$ if $p\geq 3$.\\

Thus we have $N_{0,n}(\omega ;p)=0$, $p\leq \omega \leq p+1$ if $p\geq n+1$
and so on.\\

A very brief introduction to the $p$-adic integrals are presented as
follows: 

Let $\Psi :\mathbb{Z}_{p}\rightarrow \mathbb{C}_{p}$ be a uniformly
differential function on $\mathbb{Z}_{p}$. The Volkenborn integral or the $p$%
-adic bosonic integral is given by%
\begin{equation}
	\int\limits_{\mathbb{Z}_{p}}\Psi \left( \omega \right) d\mu _{1}\left(
	\omega \right) =\underset{m\rightarrow \infty }{\lim }p^{-m}%
	\sum_{v=0}^{p^{m}-1}\Psi \left( v\right) ,  \label{Vi}
\end{equation}%
where%
\begin{equation*}
	\mu _{1}\left( \omega \right) =p^{-m}
\end{equation*}%
(cf. \cite{Schikof,Volkenborn,T. Kim,MSKIM,SimsekMTJPAM}; see also the references cited in each of these
earlier works).

The fermionic $p$-adic integral is given by%
\begin{equation}
	\int\limits_{\mathbb{Z}_{p}}\Psi \left( \omega \right) d\mu _{-1}\left(
	\omega \right) =\underset{m\rightarrow \infty }{\lim }\sum_{v=0}^{p^{m}-1}%
	\left( -1\right) ^{v}\Psi \left( v\right) ,  \label{Fi}
\end{equation}%
where%
\begin{equation*}
	\mu _{-1}\left( \omega \right) =\left( -1\right) ^{\omega }
\end{equation*}%
(cf. \cite{Kim2006TMIC}, see also \cite{MSKIMjnt2009,T. Kim,SimsekMTJPAM}).

\section{Computation Formulas of the Apostol-Bernoulli Polynomials, Eulerian Numbers and Polynomials} \label{section2}

We give some computation formulas of the Apostol-Bernoulli polynomials.
These formulas involving the Stirling numbers of the second kind, the
Bernstein polynomials and the array polynomials. By applying $p$-adic
integrals to these formulas, and using the Witt identities for the Bernoulli
and Euler numbers, we give some identities and certain family of finite sums.

By combining (\ref{ApstlBernNumbr}) with (\ref{ABP}), we have%
\begin{equation*}
	\sum_{n=0}^{\infty }\mathcal{B}_{n}(\omega ;\rho )\frac{t^{n}}{n!}%
	=\sum_{n=0}^{\infty }\sum\limits_{j=0}^{n}\mathcal{B}_{j}\left( \rho \right) 
	\frac{\omega ^{n-j}t^{n}}{j!\left( n-j\right) !}.
\end{equation*}%
Thus, the coefficients of $\frac{t^{n}}{n!}$ on both sides of the previous equation are equalized yields (\ref{Ap-i0}).

Joining (\ref{Ap-i0}) with (\ref{B1}) and (\ref{cBBF}), for $n\in \mathbb{N}_{0}$, we have the following known result:

\begin{equation}
	\mathcal{B}_{n}(\omega ;\rho
	)=\sum\limits_{j=0}^{n}\sum\limits_{s=0}^{j-1}\left( -1\right) ^{j-1}\frac{%
		\binom{n}{j}}{\binom{j-1}{s}}\frac{j}{\left( \rho -1\right) ^{j}}%
	s!B_{s}^{j-1}\left( \rho \right) S_{2}\left( j-1,s\right) \omega ^{n-j}.
	\label{ABP.F}
\end{equation}

Using (\ref{ABP}), we also get%
\begin{equation}
	\sum_{m=0}^{\infty }\mathcal{B}_{m}(\omega ;\rho )\frac{t^{m}}{m!}=\frac{t}{%
		\rho -1}e^{t\omega }\left( 1+\sum_{n=1}^{\infty }\left( \frac{\rho }{%
		1-\rho }\right) ^{n}\left( e^{t}-1\right) ^{n}\right) .  \label{fe}
\end{equation}%
Combining (\ref{fe}) with (\ref{arP}), we get%
\begin{equation*}
	\sum_{m=0}^{\infty }\mathcal{B}_{m}(\omega ;\rho )\frac{t^{m}}{m!}=\frac{t}{%
		\rho -1}\left( \sum_{m=0}^{\infty }\frac{\omega ^{m}t^{m}}{m!}%
	+\sum_{m=0}^{\infty }\sum_{n=1}^{m}\left( \frac{\rho }{1-\rho }\right)
	^{n}n!S_{n}^{m}(\omega )\frac{t^{m}}{m!}\right) .
\end{equation*}%
Therefore%
\begin{equation*}
	\sum_{m=0}^{\infty }\mathcal{B}_{m}(\omega ;\rho )\frac{t^{m}}{m!}=\frac{1}{%
		\rho -1}\left( \sum_{m=0}^{\infty }m\left( \omega
	^{m-1}+\sum_{n=1}^{m-1}\left( \frac{\rho }{1-\rho }\right)
	^{n}n!S_{n}^{m-1}(\omega )\right) \frac{t^{m}}{m!}\right) .
\end{equation*}%
Comparing the coefficients of $\frac{t^{m}}{m!}$ on both sides of the last
equation, we arrive at the following theorem:
\begin{theorem}
	Let $m\in \mathbb{N}$ with $n>1$. Then we have%
	\begin{equation*}
		\mathcal{B}_{m}(\omega ;\rho )=\frac{m}{\rho -1}\left(
		\omega^{m-1}+\sum_{n=1}^{m-1}\left( \frac{\rho }{1-\rho }\right)
		^{n}n!S_{n}^{m-1}(\omega )\right) .
	\end{equation*}
\end{theorem}
Combining (\ref{fe}) with (\ref{Sn2}), we get%
\begin{equation*}
	\sum_{m=0}^{\infty }\mathcal{B}_{m}(\omega ;\rho )\frac{t^{m}}{m!}=\frac{t}{%
		\rho -1}\left( \sum_{m=0}^{\infty }\frac{\omega ^{m}t^{m}}{m!}%
	+\sum_{m=0}^{\infty }\sum_{v=0}^{m}\binom{m}{v}\omega
	^{m-v}\sum_{n=1}^{v}\left( \frac{\rho }{1-\rho }\right) ^{n}n!S_{2}(v,n)%
	\frac{t^{m}}{m!}\right) .
\end{equation*}%
Therefore%
\begin{eqnarray*}
	&&	\sum_{m=0}^{\infty }\mathcal{B}_{m}(\omega ;\rho )\frac{t^{m}}{m!}\\&=&\frac{1}{%
		\rho -1}\left( \sum_{m=0}^{\infty }m\left( \omega ^{m-1}+\sum_{v=0}^{m-1}%
	\binom{m-1}{v}\omega ^{m-v-1}\sum_{n=1}^{v}\left( \frac{\rho }{1-\rho }%
	\right) ^{n}n!S_{2}(v,n)\right) \frac{t^{m}}{m!}\right) .
\end{eqnarray*}%
Comparing the coefficients of $\frac{t^{m}}{m!}$ on both sides of the last
equation, we arrive at the following theorem:

\begin{theorem}
	Let $m\in \mathbb{N}$ with $n>1$. Then we have
	\begin{equation*}
		\mathcal{B}_{m}(\omega ;\rho )=\frac{m}{\rho -1}\left( \omega
		^{m-1}+\sum_{v=0}^{m-1}\binom{m-1}{v}\omega ^{m-v-1}\sum_{n=1}^{v}\left( 
		\frac{\rho }{1-\rho }\right) ^{n}n!S_{2}(v,n)\right) .
	\end{equation*}
\end{theorem}
By using (\ref{aen}), we get%
\begin{equation*}
	\sum_{m=0}^{\infty }\mathcal{E}_{m}(\rho )\frac{t^{m}}{m!}=\frac{2}{1-\rho }%
	\left( 1+\sum_{n=1}^{\infty }(-1)^{n}\left( \frac{\rho }{1-\rho }\right)
	^{n}\left( e^{t}+1\right) ^{n}\right) .
\end{equation*}%
Combining the above equation with following generating function%
\begin{equation*}
	\left( e^{t}+1\right) ^{n}=n!\sum_{m=0}^{\infty }y_{1}(m,n;1)\frac{t^{m}}{m!}
\end{equation*}%
(cf. \cite[Eq. (8)]{SimsekAADM2018}), we get%
\begin{eqnarray*}
	\mathcal{E}_{0}(\rho )+\sum_{m=1}^{\infty }\mathcal{E}_{m}(\rho )\frac{t^{m}%
	}{m!} &=&\frac{2}{1-\rho }\sum_{n=1}^{\infty }(-1)^{n}\left( \frac{\rho }{%
		1-\rho }\right) ^{n}n!y_{n}(0,n;1)+ \\
	&&+\frac{2}{1-\rho }\sum_{m=1}^{\infty }\sum_{n=1}^{\infty }(-1)^{n}\left( 
	\frac{\rho }{1-\rho }\right) ^{n}n!y_{n}(m,n;1)\frac{t^{m}}{m!},
\end{eqnarray*}%
where assuming that $\left\vert \frac{\rho }{1-\rho }\right\vert <1$, we have%
\begin{equation*}
	\mathcal{E}_{0}(\rho )+\sum_{m=1}^{\infty }\mathcal{E}_{m}(\rho )\frac{t^{m}%
	}{m!}=\frac{2}{1+\rho }+\frac{2}{1-\rho }\sum_{m=1}^{\infty
	}\sum_{n=1}^{\infty }(-1)^{n}\left( \frac{\rho }{1-\rho }\right)
	^{n}n!y_{1}(m,n;1)\frac{t^{m}}{m!}.
\end{equation*}%
Comparing the coefficients of $\frac{t^{m}}{m!}$ on both sides of the above
equation, we arrive at the following theorem:

\begin{theorem}
	Let $m\in \mathbb{N}$. Then we have%
	\begin{equation}
		\mathcal{E}_{m}(\rho )=\frac{2}{1-\rho }\sum_{n=1}^{\infty }(-1)^{n}\left( 
		\frac{\rho }{1-\rho }\right) ^{n}n!y_{1}(m,n;1),  \label{cen.}
	\end{equation}%
	where $\left\vert \frac{\rho }{1-\rho }\right\vert <1$.
\end{theorem}

Since%
\begin{equation*}
	y_{1}(m,n;1)=\frac{1}{n!}\sum_{h=0}^{n}\binom{n}{h}h^{m}
\end{equation*}%
(cf. \cite{SimsekAADM2018}), (\ref{cen.}) reduces to the following corollary:

\begin{corollary}
	Let $m\in \mathbb{N}$. Then we have%
	\begin{equation*}
		\mathcal{E}_{m}(\rho )=\frac{2}{1-\rho }\sum_{n=1}^{\infty
		}\sum_{h=0}^{n}(-1)^{n}\binom{n}{h}h^{m}\left( \frac{\rho }{1-\rho }\right)
		^{n},
	\end{equation*}%
	where $\left\vert \frac{\rho }{1-\rho }\right\vert <1$.
\end{corollary}

Since%
\begin{equation*}
	\sum_{h=0}^{n}\binom{n}{h}h^{m}=\frac{d^{m}}{dt^{m}}\left\{ \left(
	e^{t}+1\right) ^{n}\right\} \left\vert _{t=0}\right. ,
\end{equation*}%
(cf. \cite{glomberg,SimsekAADM2018}), we get the following
result:
\begin{corollary}
	Let $m\in \mathbb{N}$. Then we have%
	\begin{equation*}
		\mathcal{E}_{m}(\rho )=\frac{2}{1-\rho }\sum_{n=1}^{\infty }(-1)^{n}\left( 
		\frac{\rho }{1-\rho }\right) ^{n}\frac{d^{m}}{dt^{m}}\left\{ \left(
		e^{t}+1\right) ^{n}\right\} \left\vert _{t=0}\right. ,
	\end{equation*}%
	where $\left\vert \frac{\rho }{1-\rho }\right\vert <1$.
\end{corollary}

Using (\ref{en-1}), we get%
\begin{equation}
	A_{0}(\rho )+\sum_{m=1}^{\infty }A_{m}(\rho )\frac{t^{m}}{m!}%
	=1+\sum_{n=1}^{\infty }\left( \frac{\rho }{1-\rho }\right) ^{n}\left(
	e^{t(1-\rho )}-1\right) ^{n}.  \label{fe-1}
\end{equation}%
By combining (\ref{fe-1}) and (\ref{Sn2}), we get%
\begin{equation*}
	A_{0}(\rho )+\sum_{m=1}^{\infty }A_{m}(\rho )\frac{t^{m}}{m!}%
	=1+\sum_{m=0}^{\infty }\sum_{n=1}^{m}\left( \frac{\rho }{1-\rho }\right)
	^{n}n!(1-\rho )^{m}S_{2}(m,n)\frac{t^{m}}{m!}.
\end{equation*}%
Comparing the coefficients of $\frac{t^{m}}{m!}$ on both sides of the above
equation yields
Let $m\in \mathbb{N}$. Then we have%
\begin{equation}
	A_{m}(\rho )=\sum_{n=1}^{m}\left( \frac{\rho }{1-\rho }\right) ^{n}n!(1-\rho
	)^{m}S_{2}(m,n).  \label{cap}
\end{equation}

Combining (\ref{cap}) with (\ref{cBBF}), we arrive at the following
corollary:

\begin{theorem}
	Let $m\in \mathbb{N}$. Then we have%
	\begin{equation*}
		A_{m}(\rho )=\sum_{n=1}^{m}\frac{n!}{\binom{m}{n}}B_{n}^{m}(\rho )S_{2}(m,n).
	\end{equation*}
\end{theorem}

By combining (\ref{FEn}) and (\ref{en-1}), we have the following functional
equation:%
\begin{equation*}
	F_{a}(t,\rho )=F_{N}\left( (1-\rho )t,\frac{1}{\rho }\right) .
\end{equation*}

Using this equation, we get%
\begin{equation*}
	A_{n}(\rho )=\left\{ 
	\begin{array}{cc}
		1 & \text{for }n=0 \\ 
		\rho (1-\rho )^{n}\sum\limits_{j=0}^{n}\left( 
		\begin{array}{c}
			n \\ 
			j%
		\end{array}%
		\right) H_{j}\left( \frac{1}{\rho }\right) & \text{for }n>0.%
	\end{array}%
	\right. 
\end{equation*}

By combining (\ref{ApstlBernNumbr}) and (\ref{en-1}), we have the following
functional equation:%
\begin{equation*}
	-tF_{a}(t,\rho )=K_{n}\left( (1-\rho )t,\rho \right) .
\end{equation*}%
Using this equation, for $n\in \mathbb{N}$, we have%
\begin{equation}
	A_{n-1}(\rho )=-\frac{(1-\rho )^{n}}{n}\mathcal{B}_{n}(\rho ).  \label{aEi}
\end{equation}

\begin{remark}
	Combining \textup{(\ref{cap})} with \textup{(\ref{B1})}, we arrive at \textup{(\ref{aEi})}.
\end{remark}

By combining (\ref{AnPl}), and (\ref{B1}) or using (\ref{aEi}), we give some
values of the polynomials $A_{n}(\rho )$ as follows:
\begin{eqnarray*}
	A_{1}(\rho ) &=&\frac{\left( \rho -1\right) ^{2}}{-2\rho }\mathcal{B}%
	_{2}\left( \rho \right) =A_{1,0}=1, \\
	A_{2}(\rho ) &=&\frac{\left( \rho -1\right) ^{3}}{3\rho }\mathcal{B}%
	_{3}\left( \rho \right) =A_{2,0}\rho +A_{2,1}=\rho +1 \\
	A_{3}(\rho ) &=&\frac{\left( \rho -1\right) ^{4}}{-4\rho }\mathcal{B}%
	_{4}\left( \rho \right) =A_{3,0}\rho ^{2}+A_{3,1}\rho +A_{3,2}=\rho
	^{2}+4\rho +1 \\
	A_{4}(\rho ) &=&\frac{\left( \rho -1\right) ^{5}}{5\rho }\mathcal{B}%
	_{5}\left( \rho \right) =A_{4,0}\rho ^{3}+A_{4,1}\rho ^{2}+A_{4,2}\rho
	+A_{4,3}=\alpha ^{3}+11\alpha ^{2}+11\alpha +1, \\
	A_{5}(\alpha ) &=&\frac{\left( \rho -1\right) ^{6}}{-6\rho }\mathcal{B}%
	_{6}\left( \rho \right) =A_{5,0}\rho ^{4}+A_{5,1}\rho ^{3}+A_{5,2}\rho
	^{2}+A_{5,3}\rho +A_{5,4} \\
	&=&\rho ^{4}+26\rho ^{3}+66\rho ^{2}+26\rho +1 \\
	A_{6}(\alpha ) &=&\frac{\left( \rho -1\right) ^{7}}{7\rho }\mathcal{B}%
	_{7}\left( \rho \right) =A_{6,0}\rho ^{5}+A_{6,1}\rho ^{4}+A_{6,2}\rho
	^{3}+A_{6,3}\rho ^{2}+A_{5,4}\rho +A_{6,5} \\
	&=&\rho ^{5}+57\rho ^{4}+302\rho ^{3}+302\rho ^{2}+57\rho +1,
\end{eqnarray*}%
and so on.

Boyadzhiev \cite{Boyadzhiev} gave a relation between the Apostol-Bernoulli
numbers and the geometric polynomials $W_{n}(w)$ is given as follows:%
\begin{equation}
	\mathcal{B}_{n}\left( \rho \right) =\frac{n}{\rho -1}W_{n-1}\left( \frac{%
		\rho }{1-\rho }\right)   \label{GP}
\end{equation}%
and%
\begin{equation*}
	\mathcal{B}_{n}\left( \frac{\rho }{1+\rho }\right) =-n\left( \rho +1\right)
	W_{n-1}\left( \rho \right) ,
\end{equation*}%
where $n\in \mathbb{N}$ and%
\begin{equation*}
	W_{n}(\rho )=\sum\limits_{j=0}^{n}j!S_{2}(n,j)\rho ^{j}.
\end{equation*}%
With the Euler (derivative) operator $\rho \frac{d}{d\rho }$, Boyadzhiev
also showed that%
\begin{equation*}
	\left( \rho \frac{d}{d\rho }\right) ^{k}\left\{ \frac{1}{1-\rho }\right\} =%
	\frac{1}{1-\rho }\omega _{k}\left( \frac{\rho }{1-\rho }\right)
	=\sum\limits_{v=0}^{\infty }v^{k}\rho ^{v},
\end{equation*}%
where $\left\vert \rho \right\vert <1$. Joining the above equation with (\ref%
{AnPl}), we get the following result:%
\begin{equation*}
	A_{k}(\rho )=\left( 1-\rho \right) ^{k+1}\left( \rho \frac{d}{d\rho }\right)
	^{k}\left\{ \frac{1}{1-\rho }\right\} .
\end{equation*}%
Combining the above equation with (\ref{aEi}) yields%
\begin{equation}
	-\frac{\mathcal{B}_{k+1}(\rho )}{n+1}=\left( \rho \frac{d}{d\rho }\right)
	^{k}\left\{ \frac{1}{1-\rho }\right\} .  \label{ad-1}
\end{equation}%
Joining the above equation with the following well-known equation which is
combined (\ref{FEn}) with (\ref{ABP}): 
\begin{equation}
	\mathcal{E}_{n}\left( \omega;\lambda \right) =-\frac{2}{n+1}\mathcal{B}%
	_{n+1}\left( \omega;-\lambda \right)   \label{RelationApostolEnBn}
\end{equation}%
(cf. \cite{SrivatavaChoi}), we also get%
\begin{equation}
	\mathcal{E}_{n}\left( \omega;\lambda \right) =\left( \rho \frac{d}{d\rho }\right)
	^{k}\left\{ \frac{1}{1-\rho }\right\} .  \label{ed-1}
\end{equation}%
We think that there are many different proofs of the equations (\ref{ad-1})
and (\ref{ed-1}). Some of them were also given by (cf. \cite{Bayad,Boyadzhiev,Gun,Luo,luo1,luo2,SrivatavaChoi}%
).

Combining (\ref{aEi}) with (\ref{Ap-i0}) and (\ref{ABp-i}) yields
\begin{equation}
	\sum\limits_{j=1}^{n}(-1)^{j+1}\binom{n}{j}j\left( \rho -1\right)
	^{-j}A_{j-1}(\rho )\left( \rho (\omega +1)^{n-j}-\omega ^{n-j}\right)
	=n\omega ^{n-1},  \label{epi}
\end{equation}
where $n\in \mathbb{N}$. By applying Volkenborn integral (\ref{Vi}) to (\ref{epi}) and using 
\begin{equation}
	\int\limits_{\mathbb{Z}_{p}}\omega ^{k}d\mu _{1}\left( \omega \right) =B_{k},
	\label{bv}
\end{equation}%
which is known as the Witt identity for the Bernoulli numbers, we obtain
\begin{equation*}
	\sum\limits_{j=1}^{n}(-1)^{j+1}\binom{n}{j}j\left( \rho -1\right)
	^{-j}A_{j-1}(\rho )\left( \rho B_{n-j}(1)-B_{n-j}\right) =nB_{n-1}.
\end{equation*}%
Combining the above equation with%
\begin{equation*}
	B_{n-j}(1)=B_{n-j}+\sum\limits_{k=0}^{n-j-1}\binom{n-j}{k}B_{k},
\end{equation*}%
we arrive at the following theorem:

\begin{theorem}
	Let $n\in \mathbb{N}$. Then we have%
	\begin{equation*}
		\sum\limits_{j=1}^{n}(-1)^{j+1}\binom{n}{j}j\left( \rho -1\right)
		^{-j}A_{j-1}(\rho )\left( \rho \sum\limits_{k=0}^{n-j}\binom{n-j}{k}%
		B_{k}-B_{n-j}\right) =nB_{n-1}.
	\end{equation*}
\end{theorem}

By applying the Ferminoic $p$-adic integral (\ref{Fi}) to (\ref{epi}) and
using%
\begin{equation*}
	\int\limits_{\mathbb{Z}_{p}}\omega ^{k}d\mu _{-1}\left( \omega \right)
	=E_{k},
\end{equation*}%
which is known as the Witt identity for the  Euler numbers, we get
\begin{equation*}
	\sum\limits_{j=1}^{n}(-1)^{j+1}\binom{n}{j}j\left( \rho -1\right)
	^{-j}A_{j-1}(\rho )\left( \rho E_{n-j}(1)-E_{n-j}\right) =nE_{n-1}.
\end{equation*}%
Comparing the above equation with%
\begin{equation*}
	E_{n-j}(1)=E_{n-j}+\sum\limits_{k=0}^{n-j-1}\binom{n-j}{k}E_{k},
\end{equation*}%
we arrive at the following theorem:

\begin{theorem}
	Let $n\in \mathbb{N}$. Then we have%
	\begin{equation*}
		\sum\limits_{j=1}^{n}(-1)^{j+1}\binom{n}{j}j\left( \rho -1\right)
		^{-j}A_{j-1}(\rho )\left( \rho \sum\limits_{k=0}^{n-j}\binom{n-j}{k}%
		E_{k}-E_{n-j}\right) =nE_{n-1}.
	\end{equation*}
\end{theorem}

\section{Identities, Relations and Series Representations for the Uniform $B$%
	-Splines and the Bernstein Basis Functions} \label{section3}

In this section, we give many new formulas involving uniform $B$-spline,
Apostol-Bernoulli numbers and Eulerian numbers. We also give some functional
equations and derivative equations for generating functions for the uniform $%
B$-splines and the Bernstein basis functions. Using these equations and the
Laplace transform, we study series representations for the uniform $B$%
-splines and the Bernstein basis functions.

\subsection{Relations among the uniform $B$-spline, Apostol-Bernoulli numbers
	and Eulerian numbers}

Here, by using $p$-adic integrals, we introduce very interesting results
among $B$-Spline, the Apostol-Bernoulli numbers and the Eulerian numbers.

Let's briefly give these interesting relationships as follows:

Combining (\ref{aEi}) and (\ref{Bs}) with (\ref{AnPl}) and (\ref{EnFor}), we
arrive at the following result:

\begin{theorem}
	Let $n\in \mathbb{N}_{0}$. Then we have%
	\begin{equation*}
		\mathcal{B}_{n+1}(\rho )=(-1)^{n}\left( \rho -1\right)
		^{n+1}(n+1)!\sum\limits_{j=0}^{n}N_{0,n}(j;j)\rho ^{j}.
	\end{equation*}
\end{theorem}

A relation between the uniform $B$-Spline and the Eulerian numbers is also given by
\begin{equation}
	N_{0,n}(p;p)=\frac{1}{n!}A_{n,p}.  \label{Bs3}
\end{equation}

\subsection{Relations between the uniform $B$-spline and the Bersntein basis
	functions}

By using generating functions and their functional equations and PDEs,
relations between the uniform $B$-spline and the Bersntein basis functions are
given.

\begin{theorem}
	\label{TheBS}Let $n,p\in \mathbb{N}_{0}$ and $p\leq \omega\leq p+1$. Then we have%
	\begin{equation}
		N_{0,n}(\omega;p)=\frac{1}{n!}\sum\limits_{v=0}^{n}(-1)^{v}\binom{n}{v}%
		\sum\limits_{j=0}^{p}\left( B_{j}^{v}(\omega-j)-B_{j-1}^{v}(\omega-j)\right) .
		\label{Bs-2}
	\end{equation}
\end{theorem}

\begin{proof}
	Setting the following functional equation:%
	\begin{equation}
		G_{0}(\omega,t;p)=\sum\limits_{j=0}^{p}\left( f_{\mathbb{B},j}(-t,\omega-j)-f_{\mathbb{%
				B},j-1}(-t,\omega-j)\right) e^{t}.  \label{Bs-1}
	\end{equation}%
	Combining (\ref{Bs-1}) with (\ref{BBf}) and (\ref{G-1}), we obtain%
	\begin{equation*}
		\sum\limits_{n=0}^{\infty
		}N_{0,n}(\omega;p)t^{n}=\sum\limits_{j=0}^{p}(-1)^{j}\sum\limits_{n=0}^{\infty
		}\left( B_{j}^{n}(\omega-j)-B_{j-1}^{n}(\omega-j)\right) \frac{t^{n}}{n!}%
		\sum\limits_{n=0}^{\infty }\frac{t^{n}}{n!}.
	\end{equation*}%
	Using the Cauchy product rule in the above equation, we obtain%
	\begin{equation*}
		\sum\limits_{n=0}^{\infty }N_{0,n}(\omega;p)t^{n}=\sum\limits_{n=0}^{\infty }%
		\frac{t^{n}}{n!}\sum\limits_{v=0}^{n}(-1)^{v}\binom{n}{v}\sum%
		\limits_{j=0}^{p}\left( B_{j}^{v}(\omega-j)-B_{j-1}^{v}(\omega-j)\right) .
	\end{equation*}%
	Now, by comparing the coefficients of $t^{n}$ on both sides of the above
	equation, we arrive at the assertion (\ref{Bs-2}) of Theorem \ref{TheBS}.
\end{proof}

\begin{remark}
	Equation \textup{(\ref{Bs-2})} gives us modification the Schoenbergs identity. This identity was proved by Goldman \textup{\cite[Theorem 3]{Goldman}} with the aid of generating function method.
\end{remark}

\begin{theorem}
	Let $n,p,v\in \mathbb{N}_{0}$ and $p\leq \omega\leq p+1$. Then we have%
	\begin{eqnarray}
		\frac{d^{v}}{d\omega^{v}}\left\{ N_{0,n}(\omega;p)\right\} &=&\sum\limits_{d=0}^{n}%
		\binom{n}{d}\binom{d}{v}v!\sum\limits_{m=0}^{v}(-1)^{v+n-d-m}\binom{v}{m}
		\label{A-feI} \\
		&&\times \sum\limits_{j=0}^{p}\left(
		B_{j-m}^{d-v}(\omega-j)-B_{j-m-1}^{d-v}(\omega-j)\right)  \notag
	\end{eqnarray}
\end{theorem}

\begin{proof}
	Substituting $t=-z$ into (\ref{Bs-1}), we have%
	\begin{equation}
		G_{0}(\omega,-z;p)=\sum\limits_{j=0}^{p}\left( f_{\mathbb{B},j}(z,\omega-j)-f_{\mathbb{%
				B},j-1}(z,\omega-j)\right) e^{-z}.  \label{A-fe}
	\end{equation}%
	Taking $v$th derivative of equation (\ref{A-fe}), with respect to $\omega$, we
	have the following PDE equation:%
	\begin{equation*}
		\frac{\partial ^{v}}{\partial \omega^{v}}\left\{ G_{0}(\omega,-z;p)\right\}
		=\sum\limits_{j=0}^{p}\left( \frac{\partial ^{v}}{\partial \omega^{v}}\left\{ f_{%
			\mathbb{B},j}(z,\omega-j)\right\} -\frac{\partial ^{v}}{\partial \omega^{v}}\left\{ f_{%
			\mathbb{B},j-1}(z,\omega-j)\right\} \right) e^{-z}.
	\end{equation*}%
	Combining the above equation with the following PDE for $f_{\mathbb{B}%
		,j}(t,\omega)$:%
	\begin{equation*}
		\frac{\partial ^{v}}{\partial \omega^{v}}\left\{ f_{\mathbb{B},j}(z,\omega)\right\}
		=\sum\limits_{m=0}^{v}(-1)^{v-m}\binom{v}{m}z^{v}f_{\mathbb{B},j-m}(z,\omega)
	\end{equation*}%
	(cf. \cite[Eq. (15)]{SimsekFPABERNST}), we obtain the following PDE
	for the generating function $G_{0}(\omega,-z;p)$:%
	\begin{equation*}
		\frac{\partial ^{v}}{\partial \omega^{v}}\left\{ G_{0}(\omega,-z;p)\right\}
		=\sum\limits_{j=0}^{p}\sum\limits_{m=0}^{v}(-1)^{v-m}\binom{v}{m}z^{v}\left(
		f_{\mathbb{B},j-m}(z,\omega-j)-f_{\mathbb{B},j-m-1}(z,\omega-j)\right) e^{-z}.
	\end{equation*}%
	Therefore%
	\begin{eqnarray*}
		\frac{\partial ^{v}}{\partial \omega^{v}}\left\{ G_{0}(\omega,-z;p)\right\}
		&=&\sum\limits_{j=0}^{p}\sum\limits_{m=0}^{v}(-1)^{v-m}\binom{v}{m} \\
		&&\times \left( f_{\mathbb{B},j-m}(z,\omega-j)-f_{\mathbb{B},j-m-1}(z,\omega-j)\right)
		z^{v}e^{-z}.
	\end{eqnarray*}%
	Combining the above equation with the following formula of the higher-order
	derivatives of the Bernstein basis function:%
	\begin{equation*}
		\frac{d^{v}}{d\omega^{v}}\left\{ B_{j}^{n}(\omega)\right\} =\frac{n!}{(n-v)!}%
		\sum\limits_{m=0}^{v}(-1)^{v-m}\binom{v}{m}B_{j-m}^{n-v}(\omega)
	\end{equation*}%
	(cf. \cite[Eq. (16)]{SimsekFPABERNST}), we obtain%
	\begin{eqnarray*}
		\frac{\partial ^{v}}{\partial \omega^{v}}\left\{ G_{0}(\omega,-z;p)\right\}
		=\sum\limits_{j=0}^{p}\sum\limits_{m=0}^{v}(-1)^{v-m}\binom{v}{m}z^{v}  \left( f_{\mathbb{B},j-m}(z,\omega-j)-f_{\mathbb{B},j-m-1}(z,\omega-j)\right)
		e^{-z}.
	\end{eqnarray*}%
	Therefore%
	\begin{eqnarray*}
		\sum\limits_{n=0}^{\infty }(-1)^{n}\frac{d^{v}}{d\omega^{v}}\left\{
		N_{0,n}(\omega;p)\right\} z^{n} 	&=&\sum\limits_{n=0}^{\infty }(-1)^{n}\frac{z^{n}}{n!}\sum\limits_{n=0}^{%
			\infty }\binom{n}{v}v!\sum\limits_{m=0}^{v}(-1)^{v-m}\binom{v}{m} \\
		&&\times \sum\limits_{j=0}^{p}\left(
		B_{j-m}^{n-v}(\omega-j)-B_{j-m-1}^{n-v}(\omega-j)\right) \frac{z^{n}}{n!}.
	\end{eqnarray*}%
	By using the Cauchy product rule in the above equation, we obtain%
	\begin{eqnarray*}
		\sum\limits_{n=0}^{\infty }(-1)^{n}\frac{d}{d\omega}\left\{
		N_{0,n}(\omega;p)\right\} z^{n}&=&\sum\limits_{n=0}^{\infty
		}\sum\limits_{d=0}^{n}(-1)^{n-d}\binom{n}{d}\binom{d}{v}v! \\
		&&\times \sum\limits_{m=0}^{v}(-1)^{v-m}\binom{v}{m}\sum\limits_{j=0}^{p}%
		\left( B_{j-m}^{d-v}(\omega-j)-B_{j-m-1}^{d-v}(\omega-j)\right) \frac{z^{n}}{n!}.
	\end{eqnarray*}%
	Comparing the coefficients of $z^{n}$ on both sides of the above equation,
	we arrive at the assertion (\ref{A-feI}) of Theorem.
\end{proof}

\begin{remark}
	Substituting $v=1$ into \textup{(\ref{A-feI})}, we get Theorem 4, which was proved by Goldman \textup{\cite[Theorem 4]{Goldman}}.
\end{remark}

\begin{theorem}
	Let $n,p,v\in \mathbb{N}_{0}$ and $p\leq \omega\leq p+1$. Then we have%
	\begin{eqnarray}
		N_{0,n+v}(\omega;p) &=&\frac{1}{n!(n+v)_{v}}\sum\limits_{d=0}^{n}(-1)^{v+d}\binom{%
			n}{d}\sum\limits_{m=0}^{v}(-1)^{v-m}\binom{v}{m}  \label{G-1id} \\
		&&\times \sum\limits_{j=0}^{p}\sum\limits_{c=0}^{m}B_{c}^{m}(\omega-j)\left(
		B_{j-c}^{d}(\omega-j)-B_{j-c-1}^{d}(\omega-j)\right) .  \notag
	\end{eqnarray}
\end{theorem}

\begin{proof}
	By applying the Leibnitz's formula for the $v$th derivative, with respect to $z$%
	, to equation (\ref{A-fe}), with the help of the following PDE for $f_{%
		\mathbb{B},j}(z,\omega)$:%
	\begin{equation*}
		\frac{\partial ^{m}}{\partial z^{m}}\left\{ f_{\mathbb{B},j}(z,\omega)\right\}
		=\sum\limits_{c=0}^{m}B_{c}^{m}(\omega)f_{\mathbb{B},j-c}(z,\omega)
	\end{equation*}%
	(cf. \cite[Eq. (18)]{SimsekFPABERNST}), we get the following PDE
	equation:%
	\begin{eqnarray}
		\frac{\partial ^{v}}{\partial z^{v}}\left\{ G_{0}(\omega,-z;p)\right\}
		&=&e^{-z}\sum\limits_{j=0}^{p}\sum\limits_{m=0}^{v}(-1)^{v-m}\binom{v}{m}
		\label{G-1D} \\
		&&\times \sum\limits_{c=0}^{m}B_{c}^{m}(\omega)\left( f_{\mathbb{B}%
			,j-c}(z,\omega-j)-f_{\mathbb{B},j-c-1}(z,\omega-j)\right) .  \notag
	\end{eqnarray}%
	Combining (\ref{Bs-1}) with (\ref{BBf}) and (\ref{G-1D}), we obtain%
	\begin{eqnarray*}
		&&\sum\limits_{n=0}^{\infty }(-1)^{n+v}(n+v)_{v}N_{0,n+v}(\omega;p)z^{n} \\
		&=&\sum\limits_{n=0}^{\infty }\frac{z^{n}}{n!}\sum\limits_{d=0}^{n}(-1)^{n-d}%
		\binom{n}{d}\sum\limits_{j=0}^{p}\sum\limits_{m=0}^{v}(-1)^{v-m}\binom{v}{m}
		\\
		&&\times \sum\limits_{c=0}^{m}B_{c}^{m}(\omega)\left(
		B_{j-c}^{d}(\omega-j)-B_{j-c-1}^{d}(\omega-j)\right) .
	\end{eqnarray*}%
	Comparing the coefficients of $z^{n}$ on both sides of the above equation,
	we arrive at the assertion (\ref{G-1id}) of Theorem.
\end{proof}

\begin{remark}
	Substituting $v=1$ into \textup{(\ref{G-1id})}, we get Theorem 5, which was proved by Goldman \textup{\cite[Theorem 5]{Goldman}}.
\end{remark}

\subsection{Series representations for the uniform $B$-spline and the
	Bernstein basis function}

Here, series representations for the uniform $B$-spline and the Bernstein
basis function with aid of Laplace transform are given.

\begin{theorem}
	Let $p\in \mathbb{N}_{0}$, $p\leq \omega+y\leq p+1$ and $\omega>0$, $y>0$. Then we have%
	\begin{equation}
		\sum\limits_{n=0}^{\infty }N_{0,n}(\omega;p)\frac{n!}{y^{n+2}}=\sum		\limits_{j=0}^{p}(-1)^{j} \frac{(\omega+y-j)^{j-1}}{(j-\omega)^{j+1}}. \label{sr}
	\end{equation}
\end{theorem}

\begin{proof}
	The following functional equation is derived from (\ref{G-1}):%
	\begin{equation*}
		G_{0}(\omega+y,t;p)e^{-yt}=\sum\limits_{j=0}^{p}(-1)^{j}\left( \frac{%
			(\omega+y-j)^{j}t^{j}}{j!}+\frac{(\omega+y-j)^{j-1}t^{j-1}}{(j-1)!}\right) e^{-(j-\omega)t},
	\end{equation*}%
	where $p\in \mathbb{N}_{0}$, $p\leq \omega+y\leq p+1$ and $\omega>0$, $y>0$. Applying
	the Laplace transform to the above functional equation, we get%
	\begin{eqnarray*}
		&&\sum\limits_{n=0}^{\infty }N_{0,n}(\omega;p)\int\limits_{0}^{\infty
		}t^{n}e^{-yt}dt \\
		&=&\sum\limits_{j=0}^{p}(-1)^{j}\left( \frac{(\omega+y-j)^{j}}{j!}%
		\int\limits_{0}^{\infty }t^{j}e^{-t(j-\omega)}dt+\frac{(\omega+y-j)^{j-1}}{(j-1)!}%
		\int\limits_{0}^{\infty }t^{j-1}e^{-t(j-\omega)}dt\right) .
	\end{eqnarray*}%
	Joining the above equation with the following well-known identity
	\begin{equation*}
		\int\limits_{0}^{\infty }t^{j}e^{-t}dt=j!,
	\end{equation*}%
	we obtain
	\begin{equation*}
		\sum\limits_{n=0}^{\infty }N_{0,n}(\omega;p)\frac{n!}{y^{n+1}}=\sum%
		\limits_{j=0}^{p}(-1)^{j}\left( \frac{(\omega+y-j)^{j}}{(j-\omega)^{j+1}}+\frac{%
			(\omega+y-j)^{j-1}}{(j-\omega)^{j}}\right) .
	\end{equation*}%
	After some elementary calculations, proof of theorem is completed.
\end{proof}
Substituting $y=0$ into (\ref{sr}) yields the following result:

\begin{corollary}
	Let $n,p\in \mathbb{N}_{0}$ and $p\leq \omega\leq p+1$. Then we have%
	\begin{equation*}
		\sum\limits_{n=0}^{\infty }n!N_{0,n}(\omega;p)=\sum		\limits_{j=0}^{p}(-1)^{j} \frac{(\omega+1-j)^{j-1}}{(j-\omega)^{j+1}}
	\end{equation*}
\end{corollary}

\begin{theorem}
	Let $n,p\in \mathbb{N}_{0}$ and $p\leq \omega\leq p+1$ and $y>0$. Then we have%
	\begin{equation*}
		\sum\limits_{n=0}^{\infty }\frac{n!N_{0,n}(\omega;p)}{(y+1)^{n+1}}%
		=\sum\limits_{n=0}^{\infty }(-1)^{n}\sum\limits_{j=0}^{p}\frac{%
			B_{j}^{n}(\omega-j)-B_{j-1}^{n}(\omega-j)}{y^{n+1}}.
	\end{equation*}%
	
\end{theorem}

\begin{proof}
	The following functional equation is derived from (\ref{Bs-1}):%
	\begin{equation*}
		G_{0}(\omega,t;p)e^{-(1+y)t}=\sum\limits_{j=0}^{p}\left( f_{\mathbb{B}%
			,j}(-t,\omega-j)-f_{\mathbb{B},j-1}(-t,\omega-j)\right) e^{-yt},
	\end{equation*}
	where $y>0$. By applying the Laplace transform to the above functional equation, we get
	\begin{eqnarray*}
		&&\sum\limits_{n=0}^{\infty }N_{0,n}(\omega;p)\int\limits_{0}^{\infty
		}t^{n}e^{-(1+y)t}dt \\
		&=&\sum\limits_{j=0}^{p}
		\sum\limits_{n=0}^{\infty}
		(-1)^{n}\left( \frac{	B_{j}^{n}(\omega-j)-B_{j-1}^{n}(\omega-j)}{n!}%
		\int\limits_{0}^{\infty }t^{n}e^{-ty}dt\right).
	\end{eqnarray*}%
	Since the rest of proof of this theorem is similar to the that of (\ref{sr}), we skip this proof here.
\end{proof}

\section{Conclusions}

In this article, the results and how they were obtained are discussed
together with their methods. For this purpose, by applying $p$-adic
integrals to the identities found for special polynomials, new formulas were
given that maybe serve as a resource for researchers on the subject. In
addition, by using the differential and functional equations of the
generating functions, some novel formulas for special numbers and special
polynomials were found. In addition, thanks to these methods, some formulas
were given for finite sums containing these special numbers and polynomials.
Derivative formulas of the $B$-spline curves were found by using the higher
order derivative formula of Benstein base functions. Moreover, by applying
the Laplace transform to the new generating functional equations we found,
infinite series containing $B$-spline curves are given.

It is planned in the near future to develop new mathematical models and
different applications using functional balances and differential equations
of the generator functions obtained by blending $B$-spline curves and
Bernstein basis functions.




\end{document}